\input louC.sty

\documentclass[a4paper,draft]{amsproc}
\usepackage{amsrefs}
\usepackage{amsmath}
\usepackage{mathrsfs}
\usepackage{amssymb}
\usepackage{amscd} 

\theoremstyle{plain}
 \newtheorem{thm}{\textbf{Theorem}}[section]
 
 \newtheorem{lem}{\textbf{Lemma}}[section]
 
\theoremstyle{definition}

\theoremstyle{remark}
 \newtheorem{rem}{\textbf{Remark}}[section]
 \numberwithin{equation}{section}

\renewcommand{\leq}{\leqslant}
\renewcommand{\geq}{\geqslant}

\title{Reducing Subspaces of de Branges-Rovnyak Spaces}

\subjclass[2010]{Primary 47B32}

\author[Chu]{\bfseries Cheng Chu}

\address{
Department of Mathematics \\ 
Vanderbilt University  \\ 
Nashville, Tennessee \\
USA}
\email{cheng.chu@vanderbilt.edu}

\begin{document}

\vspace{18mm}
\setcounter{page}{1}
\thispagestyle{empty}

\begin{abstract}
For $b\in H^\infty_1$, the closed unit ball of $H^\infty$, the de Branges-Rovnyak spaces $\mathcal{H}(b)$ is a Hilbert space contractively contained in the Hardy space $H^2$ that is invariant by the backward shift operator $S^*$. We consider the reducing subspaces of the operator $S^{*2}|_{\mathcal{H}(b)}$.

When $b$ is an inner function, $S^{*2}|_{\mathcal{H}(b)}$ is a truncated Toepltiz operator and its reducibility was characterized by Douglas and Foias using model theory. We use another approach to extend their result to the case where $b$ is extreme. We prove that if $b$ is extreme but not inner, then $S^{*2}|_{\mathcal{H}(b)}$ is reducible if and only if $b$ is even or odd, and describe the structure of reducing subspaces.
\end{abstract}

\maketitle

\section{Introduction}  

Let $\DD$ denote the unit disk.
Let $L^2$ denote the Lebesgue space of square integrable functions on the unit circle $\TT$. The Hardy space $H^2$ is the subspace of analytic functions on $\DD$ whose Taylor coefficients are square summable. Then it can also be identified with the subspace of $L^2$ of functions whose negative Fourier coefficients vanish. 
The space of bounded analytic functions on the unit disk is denoted by $H^\infty$.
The Toeplitz operator on the Hardy space $H^2$ with symbol $f$ in $L^\infty(\DD)$ is defined by
$$T_f (h) = P(fh),$$ for $h\in H^2(\DD)$. Here $P$ be the orthogonal projections from $L^2$ to $H^2$. The unilateral shift operator on $H^2$ is $S=T_z$.

Let $A$ be a bounded operator on a Hilbert space $H$. We define the range space $\cM(A)=AH$, and endow it with the inner product
$$\langle Af, Ag \rangle_{\mathcal{M}(A)}=\langle f, g \rangle_{H},\qq f,g\in H \ominus \m{Ker}A.$$ $\cM(A)$ has a Hilbert space structure that makes $A$ a coisometry on $H$.

Let $b$ be a function in $H^\infty_1$, the closed unit ball of $H^\infty$. The de Branges-Rovnyak space $\cH(b)$ is defined to be the space $$(I-T_b T_{\bar b})^{1/2} H^2.$$
We also define the space $\cH(\bar b)$ in the same way as $\cH(b)$, but with the roles of $b$ and $\bar{b}$ interchanged, i.e. $$\cH(\bar{b})=(I-T_{\bar b} T_b)^{1/2} H^2.$$ 
The spaces $\cH(b)$ and $\cH(\bar b)$ are also called sub-Hardy Hilbert spaces (the terminology comes from the title of Sarason's book \cite{sar94}).

The space $\cH(b)$ was introduced by de Branges and Rovnyak \cite{deb-rov1}. Sarason and several others made essential contributions to the theory \cite{sar94}. A recent two-volume monograph \cite{fri16-1}, \cite{fri16-2} presents most of the main developments in this area.

There are two special cases for $\cH(b)$ spaces. If $||b||_\infty<1$, then $\cH(b)$ is just a re-normed version of $H^2$. If $b$ is an inner function, then $$\cH(b)=H^2\ominus bH^2$$ is a closed subspace of $H^2$, the so-called model space (see \cite{garros} for a brief survey).

Let $T$ be a bounded linear operator on a Hilbert space $\cH$. A closed subspace $M$ of $\cH$ is called a reducing subspace of $T$ if $TM\subset M$ and $T^*M\subset M$. If $T$ has a proper reducing subspace, $T$ is called reducible. The reducing subspaces of shift operators or multiplication operators have been studied in various function spaces: for weighted unilateral shift operators of finite multiplicity, see \cite{ste-zhu}; for multiplication operators induced by finite Blaschke products on the Bergman space, see \cite{zhu00}, \cite{guo-huang} and the references therein. 

Our motivation is the study of reducing subspaces of truncated Toeplitz operators on the model space. For an inner function $\Gt$ and $\Gvp\in L^2$, the truncated Toeplitz operator $A^\Gt_\Gvp$ with symbol $\Gvp$ is defined by 
$$
A^\Gt_\Gvp f=P_\Gt(\Gvp f),
$$
for $f$ on the dense subset $\cH(\Gt)\cap H^\infty$ of $\cH(\Gt)$. Here $P_\Gt$ is the orthogonal projection from $H^2$ to $\cH(\Gt)$.
It is known that $A_z^{\Gt}$ is always irreducible (see e.g. \cite{garros14}). A function $f\in L^2$ is called even if $f(z)=f(-z)$, for every $z\in\DD$, and $f$ is called odd if $f(z)=-f(-z)$, for every $z\in\DD$. The operator $A_z^\Gt$ is called the compressed shift operator.
The reducibility of $A_{z^2}^{\Gt}$ is characterized by Douglas and Foias \cite{dou-foi} using model theory for contractions \cite{szn-foi} as the following. 
\begin{thm}\label{inn}
The operator $A_{z^2}^{\Gt}$ is reducible if and only if either
$\Gt$ is even or there exists $\mu\in\DD$ such that
$$
\Gt(z)=p(z)\frac{z+\mu}{1+\bar{\mu}z},
$$
where $p$ is even.
\end{thm}
Recently, Li, Yang and Lu found a different proof of Theorem \ref{inn} and extended it to the case where the symbol of a truncated Toeplitz operator is a Blaschke product of order $2$ or $3$ \cite{lyl}.  

The theory of $\cH(b)$ spaces is pervaded by a fundamental dichotomy, when $b$ is an extreme point of $H^\infty_1$ and when it is not. The nonextreme case includes $||b||_\infty<1$ and the extreme case includes $b$ is an inner function. Roughly speaking, when $b$ is nonextreme, $\cH(b)$ behaves similar to $H^2$, while in the extreme case, $\cH(b)$ is more closely related to the model space. For example, the polynomials belong to $\cH(b)$ if and only if $b$ is non-extreme (see \cite{sar94}*{Chapter IV, V}).

Notice that $(A_{z^2}^{\Gt})^*=S^{*2}|_{\cH(\Gt)}$. Thus, in view of Theorem \ref{inn},  it is natural to consider reducing subspaces of $S^{*2}|_{\cH(b)}$ when $b$ is extreme. The main purpose of this paper is to characterize the reducibility of $S^{*2}|_{\cH(b)}$ on $\cH(b)$ in the extreme case and describe the reducing subspaces (Theorem \ref{M}). We also show that $X_b$ is irreducible for every $b$.

\section{Background on de Branges-Rovnyak Spaces}

In this section, we present some basic theory of de Branges-Rovnyak spaces and the results we shall use later. 

The relation between $\cH(b)$ and $\cH(\bar b)$ can be found in \cite{sar94}*{II-4}. Here we use $\la \,\, , \,\, \ra_b$ to denote the inner product of $\cH(b)$.
\begin{thm}\label{barb}
A function $f$ belongs to $\cH(b)$ if and only if $T_{\bar b}f$ belongs to $\cH(\bar b)$. If $f_1, f_2\in \cH(b)$, then
$$
\la f_1, f_2 \ra_b=\la f_1, f_2 \ra_2+\la T_{\bar b}f_1, T_{\bar b}f_2\ra_{\bar b}.
$$
\end{thm}

Let $b\in H^\infty_1$. Let $\rho=1-|b|^2$ on $\TT$ and let $H^2(\rho)$ be the closure of polynomials in $L^2(\rho)=L^2(\TT, \rho \frac{d\Gt}{2\pi})$ (we will keep using these notations in the remaining of this paper). The Cauchy transform $K_\rho$ is the mapping from $L^2(\rho)$ to $H^2$ defined by
$$
K_\rho f=P(\rho f).
$$
In the theory of $\cH(b)$ spaces, $\cH(\bar b)$ is often more amenable than $\cH(b)$ because of a representation theorem for $\cH(\bar b)$ \cite[III-2]{sar94}.
\begin{thm}\label{rho}
The operator $K_\rho$ is an isometry from $H^2(\rho)$ to $\cH(\bar b)$.
\end{thm}

The operator on $H^2(\rho)$ of multiplication by the independent variable will be denoted by $Z_\rho$. We have the intertwining relation \cite{sar94}*{III-3} 
\beq\label{int}
K_\rho Z_\rho^*=S^* K_\rho.
\eeq

The space $\cH(b)$ is invariant under $S^*=T_{\bz}$ \cite{sar94}*{II-7}, and the restriction of $S^*$ is a contraction. We use $X_b$ to denote $S^*|_{\cH(b)}$. This operator can serve as a model for a large class of Hilbert space contractions \cite{deb-rov1}, \cite{deb-rov2}.

The following identity shows the difference between $X_b$ and $S^*$ \cite{sar94}*{II-9}.
\begin{thm}\label{x^*}
Let $b\in H^\infty_1$. For every $f\in \cH(b)$,
$$
X_b^*f=Sf- \la f, S^*b\ra_b b.
$$
\end{thm}

If $x$ and $y$ are in a Hilbert space $\cH$, we shall use $x\otimes y$ to be the following rank one operator on $\cH$
$$
(x\otimes y)(f)=\langle f,y\rangle_\cH \cdot x,\q f\in \cH.
$$
It is obvious that 
$$(x\otimes y)^*=y\otimes x,$$
and if $A, B$ are bounded linear operators on $\cH$, then
$$A(x\otimes y)B=(Ax)\otimes(B^* y).$$

It could be misleading to write the identity in Theorem \ref{x^*} as $X_b^*=S-b\otimes S^*b$ because $b$ may not be in $\cH(b)$. But it is known that $S^*b\in \cH(b)$ \cite{sar94}*{II-8}, and we have
\beq\label{xx*}
I-X_b X_b^*=(S^*b)\otimes (S^*b).
\eeq

The space $\cH(b)$ is a reproducing kernel Hilbert space with kernel function: 
$$
k_w^b(z)=\frac{1-\ol{b(w)}b(z)}{1-\bw z}.
$$
When $b$ is extreme, we have the following identity (see e.g. \cite{fri16-2}*{Theorem 25.11}).
\begin{lem}\label{x*x}
Let $b$ be an extreme point in $H^\infty_1$. Then
$$
I-X_b^*X_b=k_0^{b}\otimes k_0^{b}.
$$
\end{lem}

For an inner function $\Gt$,  $S^*\Gt$ is a cyclic vector of $(A_z^\Gt)^*$. A similar result holds for extreme functions (see e.g. \cite{fri16-2}*{Section 26.6}).
\begin{thm}\label{ex}
If $b$ is extreme, then 
$$\cH(b)=\m{Span}\{S^{*n} b: n\geq 0\}.$$
\end{thm}

\section{An Equivalent Condition for the Reducibility}

In this section we first prove that $X_b$ is irreducible for every $b$. The idea in the proof will be used to study $X_b^2$.

\begin{thm}
Let $b\in H^\infty_1$. Then $X_b$ is not reducible.
\end{thm}
\begin{proof}
Suppose $X_b$ is reducible. Then $$\cH(b)=M_1\oplus_b M_2, $$ where $M_1, M_2$ are nontrivial reducing subspaces of $X_b$.

Note that for every $f\in M_1, g\in M_2$,
$$
(I-X_bX_b^*)f \perp (I-X_bX_b^*)g
$$
in $\cH(b)$.
By Lemma \ref{xx*},  $$\m{dim}((I-X_bX_b^*)\cH(b))\leq 1.$$ Then one of the two range spaces $$(I-X_bX_b^*)M_1, (I-X_bX_b^*)M_2 $$ must be ${0}$.
WLOG, we may assume $$(I-X_bX_b^*)M_1=0,$$
i.e. for every $f\in M_1$,
$$
0=(I-X_bX_b^*)f=\la f, S^*b\ra_b S^*b.
$$
Thus $f$ is orthogonal to $S^*b$ in $\cH(b)$ and then $S^*b\in M_2$. Since $M_2$ is invariant under $S^*$, we have
$$\m{Span}\{S^{*n} b: n\geq 0\}\subset M_2.$$

If $b$ is extreme, it follows from Theorem \ref{ex} that $M_2=\cH(b)$, which is a contradiction. 

If $b$ is nonextreme, then polynomials are dense in $\cH(b)$.  We see from Theorem \ref{x^*} that $M_1$ is invariant under both $S$ and $S^*$. Pick a nonzero function $h\in M_1$, then 
$$
h(z)=\sum_{j=k}^\infty h_jz^j,
$$
for some $k\geq 0$ with $h_k\neq 0$. Thus 
$$
\frac{1}{h_k}(I-S^{k+1}S^{*k+1})h=z^k\in M_1,
$$
which implies that $M_1$ contain all the polynomials. So $M_1=\cH(b)$, which is a contradiction.
\end{proof}

For the extreme case, we have the following equivalent condition for the reducibility of $X_b^2$.
\begin{thm}\label{m1m2}
Let $b$ be an extreme point in $H^\infty_1$. Then $X_b^2$ is reducible if and only if there exist complex numbers $\Ga, \Gb$, $\Ga\Gb\neq 1$, such that
for every $n,m\geq 0$, 
\beq\label{c1}
S^{*2m}(S^*b+\Ga S^{*2}b)\perp S^{*2n}(\Gb S^*b+S^{*2}b)
\eeq
in $\cH(b)$.
In this case the reducing subspaces of $X_b^2$ are given by $$\cH(b)=M_1\oplus_b M_2,$$ where
\beq\label{m1}
M_1=\m{Span}\{ S^{*2n}(S^*b+\Ga S^{*2}b): n\geq 0\}
\eeq
and
\beq\label{m2}
M_2=\m{Span}\{ S^{*2n}(\Gb S^*b+S^{*2}b): n\geq 0\}.
\eeq
\end{thm}
\begin{proof}
Suppose \eqref{c1} holds, then take $M_1, M_2$ as in \eqref{m1}, \eqref{m2}. It is clear that $M_1, M_2$ are invariant under $X_b^2$ (or $S^{*2}$) and are orthogonal in $\cH(b)$. By Theorem \ref{ex}, we have $$\cH(b)=\m{Span}\{M_1, M_2\}.$$
Thus $$\cH(b)=M_1\oplus_b M_2,$$ and $X_b^2$ is reducible.

Next we assume $X_b^2$ is reducible. Then $$\cH(b)=M_1\oplus_b M_2, $$ where $M_1, M_2$ are nontrivial reducing subspaces of $X_b^2$.
Note that for every $f\in M_1, g\in M_2$,
$$
(I-X_b^2X_b^{*2})f\in M_1, (I-X_b^2X_b^{*2})g\in M_2.
$$
Then
$$
(I-X_b^2X_b^{*2})f\perp (I-X_b^2X_b^{*2})g
$$
in $\cH(b)$.
Using \eqref{xx*}, we have
\begin{align}\label{xx*2}
I-X_b^2X_b^{*2}&=I-X_b(X_bX_b^*)X_b^*\\
\nnb&=I-X_b(I-X_bX_b^*)X_b^*\\
\nnb&=I-X_bX_b^*-X_b(S^*b\otimes S^*b)X_b^*\\
\nnb&=S^*b\otimes S^*b+S^{*2}b\otimes S^{*2}b.
\end{align}
Note that when $b$ is extreme, $S^*b$ and $S^{*2}b$ are linearly independent. Thus dim$(I-X_b^2X_b^{*2})\cH(b))=2$. Suppose one of the two range spaces $$(I-X_b^2X_b^{*2})M_1, (I-X_b^2X_b^{*2})M_2 $$ is zero, say $$(I-X_b^2X_b^{*2})M_1=0.$$ 
By \eqref{xx*2}, we see that every function in $M_1$ is orthogonal to $S^*b$ and $S^{*2} b$ in $\cH(b)$, which implies $S^*b, S^{*2}b$ are in $M_2$. Since $M_2$ is invariant for $X_b^2$, using Theorem \ref{ex} we see that
$$\cH(b)=\m{Span}\{S^{*n} b: n\geq 0\}\subset M_2.$$ This is a contradiction. Therefore, we must have
$$
\m{dim}(I-X_b^2X_b^{*2})M_1=\m{dim}(I-X_b^2X_b^{*2})M_2=1.
$$
This means, WLOG, there exist complex numbers $\Ga, \Gb$ such that
$$
(I-X_b^2X_b^{*2})M_1=\m{Span}\{S^*b+\Ga S^{*2}b \}\subset M_1, 
$$
$$
(I-X_b^2X_b^{*2})M_1=\m{Span}\{\Gb S^*b+ S^{*2}b \}\subset M_2.
$$
Since $M_1, M_2$ are invariant under $X_b^2$, we have
$$
\m{Span}\{S^{*2n}(S^*b+\Ga S^{*2}b): n\geq 0 \}\subset M_1, 
$$
$$
\m{Span}\{S^{*2n}(\Gb S^*b+ S^{*2}b): n\geq 0\}\subset M_2.
$$
Using Theorem \ref{ex}, we obtain $$\cH(b)=M_1\cup M_2,$$ and thus \eqref{m1}, \eqref{m2} hold. The relation \eqref{c1} follows from $M_1\perp_b M_2$.
Note that $\Ga\Gb\neq 1$; otherwise $M_1=M_2=0.$
\end{proof}

\section{Main Results}

In this section, we analyze the condition \eqref{c1} and characterize the reducibility of $X_b^2$ when $b$ is extreme but not inner.

\begin{lem}\label{I}
Let $b$ be an extreme point in $H^\infty_1$. Then for every $n\geq 1$,
$$
I-X_b^{*n}X_b^{n}=\sum_{j=0}^{n-1}(X_b^{*j}k_0^{b})\otimes (X_b^{*j}k_0^{b}).
$$
\end{lem}
\begin{proof}
This proof is by induction on $n$. For $n=1$, the equality is exactly the one in Lemma \ref{x*x}.
Assume that the equality holds for some $n\geq 2$. Then, using once again Lemma \ref{x*x} and the induction hypothesis, we have
\begin{align*}
&X_b^{*n}X_b^{n}=X_b^*(X_b^{*n-1}X_b^{n-1})X_b\\
=&X^*(I-\sum_{j=0}^{n-2}(X_b^{*j}k_0^{b})\otimes (X_b^{*j}k_0^{b}))X_b\\
=&X^*X-\sum_{j=0}^{n-2}X_b^*(X_b^{*j}k_0^{b})\otimes (X_b^{*j}k_0^{b})X_b\\
=&I-k_0^{b}\otimes k_0^{b}-\sum_{j=0}^{n-2}(X_b^{*(j+1)}k_0^{b})\otimes (X_b^{*(j+1)}k_0^{b})\\
=&I-\sum_{j=0}^{n-1}(X_b^{*j}k_0^{b})\otimes (X_b^{*j}k_0^{b}).
\end{align*}
\end{proof}

\begin{lem}\label{2nm}
Let $b$ be an extreme point in $H^\infty_1$ and let $f, g\in \cH(b)$. Then 
\beq\label{nm}
\la X_b^{2m}f, X_b^{2n}g\ra_b=0,
\eeq
for every $m, n\geq 0$ if and only if the following hold
\begin{enumerate}
\item for every $k\geq 0$,
\beq\label{0}
\la T_{\bar b}f, T_{\bar b}X_b^{2k} g \ra_{\bar b}=\la T_{\bar b}g, T_{\bar b}X_b^{2k} f \ra_{\bar b}=0.
\eeq
\item for every $m, n\geq 0$,
$$
\la S^{*2m}f, S^{*2n}g\ra_2=0,
$$
i.e.
there exist functions $F, G\in H^2$ and complex numbers $a_0, b_0, a_1, b_1$ such that
\beq\label{fg}
f(z)=F(z^2)(a_0+a_1 z),\q g(z)=G(z^2)(b_0+b_1 z)
\eeq
and
$$a_0\bar{b_0}+a_1\bar{b_1}=0.$$
\end{enumerate}
\end{lem}

\begin{proof}
Let
$$
f(z)=\sum_{k=0}^\infty f_k z^k\q\m{and}\q g(z)=\sum_{k=0}^\infty g_k z^k.
$$
Suppose \eqref{nm} holds. Then for $m\leq n$, we have
\beq\label{1}
0=\la X_b^{2m}f, X_b^{2n}g\ra_b=\la X_b^{*2m}X_b^{2m}f, X_b^{2n-2m}g\ra_b.
\eeq
By Lemma \ref{I}, we have
\begin{align*}
(I-X_b^{*2m}X_b^{2m})f&=f-X_b^{*2m}X_b^{2m}f\\
&=\sum_{j=0}^{2m-1}\la f, X_b^{*j}k_0^{b}\ra_b\cdot (X_b^{*j}k_0^{b})\\
&=\sum_{j=0}^{2m-1}\la X_b^j f, k_0^{b}\ra_b\cdot  (X_b^{*j}k_0^{b})\\
&=\sum_{j=0}^{2m-1}(S^{*j} f)(0)\cdot (X_b^{*j}k_0^{b})\\
&=\sum_{j=0}^{2m-1}f_j\cdot X_b^{*j}k_0^{b}.
\end{align*}
Then
$$
X_b^{*2m}X_b^{2m} f=f- \sum_{j=0}^{2m-1}f_j\cdot X_b^{*j}k_0^{b}.
$$
This together with \eqref{1} implies
\begin{align}\label{11}
\nnb0&=\la f- \sum_{j=0}^{2m-1}f_j\cdot X_b^{*j}k_0^{b}, X_b^{2n-2m}g\ra_b\\
\nnb&=\la f, X_b^{2n-2m}g\ra_b -\la \sum_{j=0}^{2m-1}f_j\cdot  X_b^{*j}k_0^{b}, X_b^{2n-2m}g\ra_b\\
\nnb&=\la f, X_b^{2n-2m}g\ra_b -\sum_{j=0}^{2m-1}f_j\cdot \la   X_b^{*j}k_0^{b}, X_b^{2n-2m}g\ra_b\\
\nnb&=\la f, X_b^{2n-2m}g\ra_b -\sum_{j=0}^{2m-1}f_j\cdot \la  k_0^{b}, X_b^{2n-2m+j}g\ra_b\\
\nnb&=\la f, X_b^{2n-2m}g\ra_b -\sum_{j=0}^{2m-1}f_j\cdot \ol{\la X_b^{2n-2m+j}g, k_0^{b} \ra_b}\\
&=\la f, X_b^{2n-2m}g\ra_b -\sum_{j=0}^{2m-1}f_j\cdot \ol{g}_{2n-2m+j}.
\end{align}
Replacing $n, m$ in \eqref{11} by $n+1, m+1$ respectively, we have
\beq\label{2}
0=\la f, X_b^{2n-2m}g\ra_b -\sum_{j=0}^{2m+1}f_j\cdot \ol{g}_{2n-2m+j}.
\eeq
Subtracting \eqref{2} by \eqref{11} implies 
\beq\label{3}
f_{2m}\ol{g}_{2n}+f_{2m+1}\ol{g}_{2n+1}=0,
\eeq
for $m\leq n$.
A similar argument shows that \eqref{3} also holds for $n\leq m$. Thus we have for every $n, m\geq 0$, the two vectors
$$(f_{2m},f_{2m+1}), (g_{2n},g_{2n+1})$$ are orthogonal in $\CC^2$.
It is easy to check $f, g$ must have the form \eqref{fg}. 
In particular, we have 
\beq\label{t1}
\la f, X_b^{2k} g \ra_2=\la g, X_b^{2k} f \ra_{2}=0, \q\m{for every}\,\,k\geq 0.
\eeq
It follows from \eqref{11} and \eqref{3} that 
$$
\la f, X_b^{2k} g \ra_b=\la g, X_b^{2k} f \ra_{b}=0, \q\m{for every}\,\,k\geq 0.
$$
This together with \eqref{t1} and Theorem \ref{barb} give \eqref{0}.

The sufficiency follows easily from the calculation in \eqref{11}.

\end{proof}

\begin{rem}
When $b$ is an inner function, $\cH({\bar b})$ is trivial and then \eqref{0} is automatically satisfied. One may expect the reducibility of $X_b^2$ is more restrictive if $b$ is not inner. We shall see it is true in the remaining of this section.
\end{rem}

When $b$ is extreme, the following Lemma will be used to calculate the inner products in \eqref{0}.
\begin{lem}\label{bb}
Let $b$ be an extreme point in $H^\infty_1$. Let $\rho=1-|b|^2$ on $\TT$. Then for every $m,n \geq 1$, 
$$
\la T_{\bar b}S^{*m}b, T_{\bar b}S^{*n}b\ra_{\bar b}=\left\{
\begin{aligned}
&-\la z^{n-m},|b|^2 \ra_2, &m<n, \\
&-\la |b|^2, z^{m-n}\ra_2, &m>n,\\
&1-||b||_2^2, &m=n.
\end{aligned}
\right.     
$$
\end{lem}
\begin{proof}
Suppose $m\leq n$.
Using the intertwining relation \eqref{int}, we can easily get
$$
K_\rho Z_\rho^{*n}=S^{*n}K_\rho.
$$
Thus 
\begin{align*}
K_\rho Z_\rho^{*n}1=&S^{*n}K_\rho 1=S^{*n}P(\rho)=S^{*n}P(1-|b|^2)\\
=&-S^{*n}P(|b|^2)=-S^{*n}T_{\bar b} b=-T_{\bar b} S^{*n}b.
\end{align*}
By Theorem \ref{rho}, we have 
$$
\la T_{\bar b}S^{*m}b, T_{\bar b}S^{*n}b\ra_{\bar b}=\la K_\rho Z_\rho^{*m}1, K_\rho Z_\rho^{*n}1 \ra_{\bar b}= \la  Z_\rho^{*m}1, Z_\rho^{*n}1 \ra_{L^2(\rho)}.
$$
If $b$ is extreme, then $H^2(\rho)=L^2(\rho)$ \cite{sze20}, which implies $Z_\rho$ is a unitary operator.
Then 
\begin{align*}
\la T_{\bar b}S^{*m}b, T_{\bar b}S^{*n}b\ra_{\bar b}=& \la  Z_\rho^{*m}1, Z_\rho^{*n}1 \ra_{L^2(\rho)}=\la  Z_\rho^{n-m}1, 1 \ra_{L^2(\rho)}=\la z^{n-m}, 1\ra_{L^2(\rho)}\\
=&\la z^{n-m},1\ra_2-\la z^{n-m},|b|^2 \ra_2\\
=&
\left\{
\begin{aligned}
&-\la z^{n-m},|b|^2 \ra_2, &m<n, \\
&1-||b||^2_2, &m=n.
\end{aligned}
\right.       
\end{align*}
\end{proof}

We also need the following three elementary results. 

\begin{lem}\label{lim}
Let $b\in H^\infty_1$. Then
$$
\lim_{n\to\infty}\la z^n, |b|^2\ra_2=0.
$$
\end{lem}
\begin{proof}
Let $$b(z)=\sum_{k=0}^\infty b_k z^k.$$ Then
\begin{align*}
|\la z^n, |b|^2\ra_2&=|\la z^nb, b\ra_2|=|\sum_{k=0}^\infty b_k\ol{b_{n+k}}|\\
&\leq (\sum_{k=0}^\infty|b_k|^2)^{1\over 2}(\sum_{k=0}^\infty|b_{n+k}|^2)^{1\over 2}=||b||_2(\sum_{k=0}^\infty|b_{n+k}|^2)^{1\over 2}.
\end{align*}
Since $||b||_2\leq 1$, we have 
$$
\lim_{n\to\infty}|\la z^n, |b|^2\ra_2|\leq (\lim_{n\to\infty}\sum_{k=0}^\infty|b_{n+k}|^2)^{1\over 2}=0.
$$
\end{proof}

\begin{lem}\label{even}
Let $b\in H^\infty$. Then $|b|^2$ is even if and only if $b$ is even or odd. 
\end{lem}
\begin{proof}
Let $b(z)=b_0(z)+zb_1(z)$, where $b_0, b_1$ are even functions. Then $|b|^2$ is even if and only if 
$$
|b_0(z)+zb_1(z)|^2=|b_0(z)-zb_1(z)|^2,
$$
which is equivalent to
$b_0\ol{zb_1}\equiv 0$. Then the conclusion follows easily.
\end{proof}

\begin{lem}\label{an}
Let $\Ga,\Gb\in\CC$ with $\Ga\Gb\neq 0 \,\,\m{or}\,\, 1$. Let $\{a_n\}_{n=0}^\infty$ be a sequence of complex numbers but not the zero sequence. Suppose $$\lim_{n\to\infty}a_n=0$$ and for every $n\geq 1$, the following conditions hold.
\beq\label{an1}
a_{2n+1}+(\Ga+\bar{\Gb})a_{2n}+\Ga\bar{\Gb}a_{2n-1}=0,\eeq
\beq\label{an2}
a_{2n+1}+(\frac{1}{\bar{\Ga}}+\frac{1}{\Gb})a_{2n}+\frac{1}{\bar{\Ga}\Gb}a_{2n-1}=0.\eeq
Then we have either $$\Gb=-\bar{\Ga},\q\m{and}\q a_{2n-1}=0,\q \m{for every}\,\, n\geq 1$$
or
$$|\Ga|=|\Gb|=1.$$
\end{lem}
\begin{proof}
Subtracting \eqref{an2} from \eqref{an1}, we have
\beq\label{an3}
(\Ga+\bar{\Gb}-\frac{1}{\bar{\Ga}}-\frac{1}{\Gb})a_{2n}+(\Ga\bar{\Gb}-\frac{1}{\bar{\Ga}\Gb})a_{2n-1}=0.
\eeq
Since $\{a_n\}_{n=0}^\infty$ is nonzero, we have the following four cases.

Case I: $\Ga+\bar{\Gb}-\frac{1}{\bar{\Ga}}-\frac{1}{\Gb}=\Ga\bar{\Gb}-\frac{1}{\bar{\Ga}\Gb}=0.$
Then we have $|\Ga\Gb|=1$ and 
\begin{align*}
0=&\Ga+\bar{\Gb}-\frac{1}{\bar{\Ga}}-\frac{1}{\Gb}=\frac{1}{\bar{\Ga}}(1-|\Ga|^2)+\frac{1}{\Gb}(1-|\Gb|^2)\\
=&\frac{1}{\bar{\Ga}}(1-|\Ga|^2)+\frac{1}{\Gb}(1-\frac{1}{|\Ga|^2})=\frac{1-|\Ga|^2}{\bar{\Ga}}(1-\frac{1}{\Ga\Gb}).
\end{align*}
Thus $|\Ga|=|\Gb|=1.$

Case II: $\Ga+\bar{\Gb}=\frac{1}{\bar{\Ga}}+\frac{1}{\Gb}$ and $a_{2n-1}=0$, for every $n\geq 1$. Then \eqref{an1} implies $\Gb=-\bar{\Ga}.$

Case III: $\Ga\bar{\Gb}=\frac{1}{\bar{\Ga}\Gb}$ and $a_{2n}=0$, for every $n\geq 1$. Then $|\Ga\Gb|=1$ and by \eqref{an1}, we have
$$
|a_{2n+1}|=|\Ga\Gb|\cdot|a_{2n-1}|=|a_{2n-1}|.
$$
Since $a_n$ tends to $0$, we have $a_{2n-1}=0$ and thus $\{a_n\}_{n=0}^\infty$ is the zero sequence, which contradicts the assumption. 

Case IV: $\Ga+\bar{\Gb}-\frac{1}{\bar{\Ga}}-\frac{1}{\Gb}\neq 0$ and $\Ga\bar{\Gb}-\frac{1}{\bar{\Ga}\Gb}\neq 0.$
Then by \eqref{an3}, 
$$
a_{2n}=\frac{\frac{1}{\bar{\Ga}\Gb}-\Ga\bar{\Gb}}{\Ga+\bar{\Gb}-\frac{1}{\bar{\Ga}}-\frac{1}{\Gb}} a_{2n-1}=\frac{1-|\Ga\Gb|^2}{\Gb|\Ga|^2+\bar{\Ga}|\Gb|^2-\Gb-\bar{\Ga}} a_{2n-1}.
$$
Put this in \eqref{an1}, we have
\begin{align*}
a_{2n+1}&=-(\Ga+\bar{\Gb})a_{2n}-\Ga\bar{\Gb}a_{2n-1}\\
&=-\Big( (\Ga+\bar{\Gb})\frac{1-|\Ga\Gb|^2}{\Gb|\Ga|^2+\bar{\Ga}|\Gb|^2-\Gb-\bar{\Ga}}+\Ga\bar{\Gb} \Big) a_{2n-1}\\
&=\frac{\Ga|\Gb|^2+\bar{\Gb}|\Ga|^2-\Ga-\bar{\Gb}}{\Gb|\Ga|^2+\bar{\Ga}|\Gb|^2-\Gb-\bar{\Ga}}a_{2n-1}.
\end{align*}
Thus $|a_{2n+1}|=|a_{2n-1}|$ and, similar to Case III, $a_{2n-1}=0$, for every $n\geq 1$. From \eqref{an1}, \eqref{an2}, we see that
either $a_{2n}=0$, for every $n\geq 1$ or $\Ga+\bar{\Gb}=\frac{1}{\bar{\Ga}}+\frac{1}{\Gb}=0.$ They are both excluded by the assumptions.
\end{proof}

Now we are ready to prove the main Theorem. 

\begin{thm}\label{M}
Let $b$ be an extreme point in $H^\infty_1$. If $b$ is not an inner function, then $X_b^2$ is reducible if and only if $b$ is even or odd. If $b$ is even, the reducing subspaces of $X_b^2$ are
$$
M=\m{Span}\{ (S^{*2n}b) (z+\Ga): n\geq 1\}
$$
with
$$
M^{\perp}=\m{Span}\{ (S^{*2n}b) (-\bar{\Ga}z+1): n\geq 1\},
$$
for all $\Ga\in\CC$.

If $b$ is odd, the reducing subspaces of $X_b^2$ are
$$
M=\m{Span}\{ S^{*2n-1}b: n\geq 1\}
$$
with
$$
M^{\perp}=\m{Span}\{ S^{*2n}b: n\geq 1\}.
$$
\end{thm}
\begin{proof}
 
{\bf Necessity.} We assume $X_b^2$ is reducible and $b$ is not inner.
Let $$b(z)=\sum_{k=0}^\infty b_k z^k.$$ 
By Theorem \ref{m1m2}, there exists $\Ga, \Gb\in\CC$ such that $\Ga\Gb\neq 1$ and \eqref{c1} holds. 
Let 
$$
f=S^*b+\Ga S^{*2}b,  \q g=\Gb S^*b+S^{*2}b.
$$
Then $f,g$ are in $\cH(b)$, and using Lemma \ref{2nm}, we have
$$
\la T_{\bar b}f, T_{\bar b}X_b^{2n} g \ra_{\bar b}=\la T_{\bar b}g, T_{\bar b}X_b^{2n} f \ra_{\bar b}=0,
$$
for every $n\geq 0$.
If $n\geq 1$, using Lemma \ref{bb}, we have
\begin{align*}
\nnb 0=&\la T_{\bar b}f,\,\, T_{\bar b}X_b^{2n} g \ra_{\bar b}\\
\nnb=&\la T_{\bar b}S^*b+\Ga T_{\bar b}S^{*2}b, \,\, \Gb T_{\bar b} (S^*)^{2n+1}b+T_{\bar b}(S^*)^{2n+2}b  \ra_{\bar b}\\
\nnb=&\bar{\Gb}\la T_{\bar b}S^*b, T_{\bar b}(S^*)^{2n+1}b\ra_{\bar b}+\la T_{\bar b}S^*b, T_{\bar b}(S^*)^{2n+2}b\ra_{\bar b}\\
\nnb&+\Ga\bar{\Gb}\la T_{\bar b}S^{*2}b, T_{\bar b}(S^*)^{2n+1}\ra_{\bar b}+\Ga \la T_{\bar b}S^{*2}b, T_{\bar b}(S^*)^{2n+2}b\ra_{\bar b}\\
\nnb=&-\bar{\Gb}\la z^{2n}, |b|^2 \ra_2 -\la z^{2n+1}, |b|^2 \ra_2 -\Ga\bar{\Gb}\la z^{2n-1}, |b|^2 \ra_2 -\Ga \la z^{2n}, |b|^2 \ra_2,
\end{align*}
which can be simplified to
\beq\label{ab1}
\la z^{2n+1}, |b|^2 \ra_2 +(\Ga+\bar{\Gb})\la z^{2n}, |b|^2 \ra_2 +\Ga\bar{\Gb}\la z^{2n-1}, |b|^2 \ra_2 =0.
\eeq

Similarly, 
$$
\la T_{\bar b}g, T_{\bar b}X_b^{2n} f \ra_{\bar b}=0
$$
implies
\beq\label{ab2}
\bar{\Ga}\Gb\la z^{2n+1}, |b|^2 \ra_2 +(\bar{\Ga}+{\Gb})\la z^{2n}, |b|^2 \ra_2 +\la z^{2n-1}, |b|^2 \ra_2=0.
\eeq

If $\Ga=\Gb=0$, then \eqref{ab2} implies for every $n\geq 1$,
$$
0=\la z^{2n-1}, |b|^2\ra_2=\la |b|^2, \bz^{2n-1} \ra_2.
$$
This means $b$ is even or odd by Lemma \ref{even}.

If $\Ga=0$ and $\Gb\neq 0$, using \eqref{ab1}, \eqref{ab2}, we have
$$
\la z^{2n+1}, |b|^2 \ra_2 +\bar{\Gb}\la z^{2n}, |b|^2 \ra_2 =0,
$$
and 
$$
{\Gb}\la z^{2n}, |b|^2 \ra_2 +\la z^{2n-1}, |b|^2 \ra_2 =0.
$$
Thus
$$
|\la z^{2n+1}, |b|^2 \ra_2 |=|\frac{\bar{\Gb}}{\Gb}\la z^{2n-1}, |b|^2 \ra_2 |=|\la z^{2n-1}, |b|^2 \ra_2 |.
$$
By Lemma \ref{lim}, we see that for every $n\geq 1$,
$
\la z^{2n-1}, |b|^2 \ra_2 =0.
$
Thus \eqref{ab1} shows that $\la z^n, |b|^2\ra_{2}=0,$ which implies $b$ is inner. A similar argument works for the case when $\Gb=0$ and $\Ga\neq 0$.

Next, suppose $\Ga\Gb\neq 0$. Rewrite \eqref{ab2} as 
\beq\label{ab3}
\la z^{2n+1}, |b|^2 \ra_2 +(\frac{1}{\bar{\Ga}}+\frac{1}{\Gb})\la z^{2n}, |b|^2 \ra_2 +\frac{1}{\bar{\Ga}\Gb}\la z^{2n-1}, |b|^2 \ra_2=0.
\eeq
Consider the sequence $\{\la z^{n}, |b|^2\ra_2\}_{n=1}^\infty$. If it is the zero sequence, then $b$ is an inner function. Otherwise by Lemma \ref{lim} and \eqref{ab1}, \eqref{ab3}, it satisfies the assumptions in Lemma \ref{an}. Then we have the following two cases: 

Case I: $\Gb=-\bar{\Ga}$, $\la z^{2n-1}, |b|^2 \ra_2 =0$, for every $n\geq 1$. 

Case II: $|\Ga|=|\Gb|=1.$

By condition $(2)$ in Lemma \ref{2nm}, we have for every $n\geq 0$, $$\la S^{*2n}f, S^{*2n}g \ra_2=0.$$  
Then
\begin{align*}
0=&\la S^{*2n}f, S^{*2n}g \ra_2=\la S^{*2n+1}b+\Ga S^{*2n+2}b, \Gb S^{*2n+1}b+S^{*2n+2}b \ra_2\\
=&\bar{\Gb}||S^{*2n+1}b||_2^2+\Ga ||S^{*2n+2}b||_2^2+\la S^{*2n+1}b, S^{*2n+2}b\ra_2+\Ga\bar{\Gb} \la S^{*2n+2}b, S^{*2n+1}b\ra_2.
\end{align*}
For simplicity, let 
$$
c_n=\la S^{*2n+1}b, S^{*2n+2}b\ra_2.
$$ 
Since 
$$
||S^{*2n+2}b||_2^2=||S^{*2n+1}b||_2^2-|b_{2n+1}|^2,
$$
we obtain
\beq\label{b1}
(\bar{\Gb}+\Ga)||S^{*2n+1}b||_2^2-\Ga|b_{2n+1}|^2+c_n+\Ga\bar{\Gb}\bar{c_n}=0.
\eeq

In Case I, $|b|^2$ is even, and Lemma \ref{even} implies that $b$ is even or odd. Thus $c_n=0$ and \eqref{b1} becomes
$\Ga|b_{2n+1}|^2=0$, which implies $b_{2n+1}=0$ and $b$ is even. 

In Case II, taking conjugate on \eqref{b1}, we get
\beq\label{b2}
({\Gb}+\bar{\Ga})||S^{*2n+1}b||_2^2-\bar{\Ga}|b_{2n+1}|^2+\bar{\Ga}\Gb c_n+\bar{c_n}=0.
\eeq
Multiplying \eqref{b1} by $\bar{\Ga}\Gb$ and using $|\Ga|=|\Gb|=1$, we have
\beq\label{b3}
(\bar{\Ga}+\Gb)||S^{*2n+1}b||_2^2-\Gb|b_{2n+1}|^2+\bar{\Ga}\Gb c_n+\bar{c_n}=0.
\eeq
By \eqref{b2}, \eqref{b3}, we have
$$
(\bar{\Ga}-\Gb)|b_{2n+1}|^2=0.
$$
Note that $\bar{\Ga}\neq \Gb$ because $\Ga\Gb\neq 1$. We see that $b_{2n+1}=0$, which means $b$ is even.
Using \eqref{ab1}, we see that if $b$ is not inner, then $\Gb=-\bar{\Ga}$.

{\bf Sufficiency.} 
Let 
$$
M_1=\m{Span}\{ S^{*2n}(S^*b+\Ga S^{*2}b): n\geq 0\}
$$
and
$$
M_2=\m{Span}\{ S^{*2n}(-\bar{\Ga} S^*b+S^{*2}b): n\geq 0\}.
$$
We show that $M_1,M_2$ are reducing subspaces of $X_b^2$ for appropriate choices of $\Ga$. By Theorem \ref{m1m2} and Lemma \ref{2nm}, we need to verify \eqref{0} and \eqref{fg} when $\Gb=-\bar{\Ga}$.

Note that $$
\la z^{2n-1}, |b|^2\ra_2=0,\q\m{for every}\,\, n\geq 1.
$$
whenever $b$ is even or odd.

For \eqref{0}, if $n\geq 1$, \eqref{0} follows from \eqref{ab1}, \eqref{ab2} and the above relation. 
When $n=0$, using Lemma \ref{bb}, we have
\begin{align*}
&\la T_{\bar b}(S^*b+\Ga S^{*2}b),\,\, T_{\bar b}(-\bar{\Ga} S^*b+S^{*2}b) \ra_{\bar b}\\
=& -\Ga||T_{\bar b}S^*b||_{\bar b}^2+\Ga ||T_{\bar b}S^{*2}b||_{\bar b}^2+\la T_{\bar b}S^*b, T_{\bar b}S^{*2}b\ra_{\bar b}-\Ga^2\la T_{\bar b}S^{*2}b, T_{\bar b}S^{*}b\ra_{\bar b}\\
=& -\Ga(1-||b||_2^2)+\Ga (1-||b||_2^2)- \la z, |b|^2\ra_2+\Ga^2 \la \bz, |b|^2\ra_2=0.
\end{align*}

If $b$ is odd and $\Ga=0$, it is obvious that \eqref{fg} holds.

If $b$ is even, then $S^{*2}b$ is also even and $$S^{*}b=z S^{*2}b.$$
We can write 
$$
S^*b+\Ga S^{*2}b=(S^{*2}b)(z+\Ga),
$$
and
$$
-\bar{\Ga} S^*b+S^{*2}b= (S^{*2}b)(-\bar{\Ga}z+1).
$$
Thus \eqref{fg} is satisfied.

\end{proof}

\bibliographystyle{plain}
\bibliography{references}

\end{document}